\documentclass[a4paper,12pt]{amsart}
\usepackage{amsmath,amsthm}
\usepackage{amsfonts,amsmath,amsthm, amssymb, mathrsfs}
\usepackage{float}
\usepackage{euscript,eufrak,verbatim}
\usepackage{graphicx}
\usepackage[usenames]{color}
\usepackage[colorlinks,linkcolor=red,anchorcolor=blue,citecolor=blue]{hyperref}
\usepackage{amsmath}
\usepackage{amsthm}
\usepackage[all]{xy}
\usepackage{graphicx} 
\usepackage{amssymb} 
\usepackage{enumerate}
\usepackage{enumitem}

\newtheorem{thm}{Theorem}[section]
\newtheorem{prop}[thm]{Proposition}

\newtheorem{lem}[thm]{Lemma}

\newtheorem{rem}[thm]{Remark}

\def\N{\mathbb{N}}

\def\N{\mathbb{N}}

\makeatletter 
\@addtoreset{equation}{section}
\numberwithin{equation}{section}

\title{Irregular set and   metric mean dimension with potential}

\author{Tianlong Zhang$^1$, Ercai Chen$^1$ and Xiaoyao Zhou$^1$*}
\address
{1.School of Mathematical Sciences and Institute of Mathematics, Ministry of Education Key Laboratory of NSLSCS, Nanjing Normal University, Nanjing, 210023, Jiangsu, P.R.China}
\email{ztl20001007@163.com}
\email{ecchen@njnu.edu.cn}
\email{zhouxiaoyaodeyouxian@126.com}
\date{}

\begin{document}

\maketitle

\renewcommand{\thefootnote}{}
\footnote{*corresponding author.}

\begin{abstract}
	Let $(X,f)$ be a dynamical system with the specification property and $\varphi$ be a continuous function. In this paper, we consider the multifractal irregular set
	\begin{align*}
		I_{\varphi}=\left\{x\in X:\lim\limits_{n\to\infty}\frac{1}{n}\sum_{i=0}^{n-1}\varphi(f^ix)\ \text{does not exist}\right\}
	\end{align*}  and show that this set is either empty  or carries full Bowen upper and lower metric mean dimension with potential.

\end{abstract}

\section{Introduction}  
This paper contributes to the study of  the information  the upper and lower metric mean dimension with potential of  multifractal irregular set carry.

In this paper, we focus	on the following framework.
Let $(X,d,f)$ be a topological dynamical system(abbr, $\rm{TDS}$), i.e. a compact space $(X,d)$ and a continuous transformation $f:X\to X$. For any continuous function $\varphi:X\to\mathbb{R}$, the space $X$ has a natural multifractal decomposition
$$X=\bigcup_{\alpha\in\mathbb{R}}K_{\alpha}\cup I_{\varphi}$$
where 
\begin{align*}
K_{\alpha}&=\left\{x\in X:\lim\limits_{n\to\infty}\frac{1}{n}\sum_{i=0}^{n-1}\varphi(f^ix)=\alpha\right\},\\
I_{\varphi}&=\left\{x\in X:\lim\limits_{n\to\infty}\frac{1}{n}\sum_{i=0}^{n-1}\varphi(f^ix)\ \text{does not exist}\right\}.
\end{align*}
In this paper, we take $I_{\varphi}$ as our  research object. $I_{\varphi}$ was called irregular set, the set of divergence points or historic set by  different researchers. In the early stages, $I_{\varphi}$ has been considered of little interest in dynamical systems and geometric measure theory. However,  the works of \cite{cx99,ff00,ffw01,flw02} have changed such attitudes. For a general dynamical system with the specification property, Chen, K\"upper and Shu \cite{cks05} proved that  the irregular set is either empty or carries full topological entropy. In 2010, Thompson  \cite{th10} extended it to  that the irregular set is either empty or carries full topological pressure for maps with the specification property.
Meanwhile, the irregular set $I_{\varphi}$ is either empty or carrying full Bowen topological in systems with the almost specification property \cite{th12} and in systems with shadowing property \cite{dot18}.
Recently, using (metric) mean dimension to describe the complexity of infinite entropy systems has attracted people's research interest. In 2021, Lima and Varandas \cite{lv21} proved that for a continuous map with the gluing orbit property on a compact metric space $I_{\varphi}$ is either an empty set or carries full topological pressure and full metric mean dimension.
Liu and Liu  \cite{ll24} gave a contribution to the study  of the metric mean dimension by proving that $I_{\varphi}$ of  a system with the weak specification property is empty or its Bowen (upper and lower) metric mean dimension coincides with the classical  (upper and lower) metric mean dimension of $(X,d,f)$.  
For a dynamical system with the shadowing property, Fory\'s-Krawiec and Oprocha give a similar result \cite{fo24}.
Recently, Tsukamoto introduced the notion of upper mean dimension with potential \cite{tsu20}. Cheng, Li and Selmi \cite{cls21} introduced the Bowen upper metric mean dimension with potential. This prompts us to use these quantities to measure the complexity of $I_{\varphi}$. 

Now, we state our main result as follows:
\begin{thm}\label{thm 1.1}
Let $(X,d,f)$ be a $\mathrm{TDS}$ satisfying the specification property, $I_{\varphi}$ be the irreugular set of $\varphi$ and $\psi\in C(X,\mathbb{R})$. If $\overline{\rm{mdim}}_M(f,X,d,\psi)<\infty$(resp. $\underline{\rm{mdim}}_M(X,f,\psi)<\infty$), then either $I_{\varphi}=\emptyset$ or
\begin{align*}
&\overline{\rm{mdim}}_M^B(f,I_{\varphi},d,\psi)=\overline{\rm{mdim}}_M(f,X,d,\psi),\\
&\underline{\rm{mdim}}_M^B(f,I_{\varphi},d,\psi)=\underline{\rm{mdim}}_M(f,X,d,\psi).
\end{align*}
\end{thm}

\section{Preliminary} 
In this section, we give some auxiliary quantities and introduce some lemmas that we will use in the proof of Theorem \ref{thm 1.1}.

Let $(X,d, f)$ be a topological dynamical system (abbr. $\mathrm{TDS}$), i.e., a compact metric space $(X,d)$ with a continuous transformation $f:X\to X$.
Given $n\in \mathbb N$, $x,y \in X$, the $n$-th Bowen metric $d_n$  on $X$ is defined by  $$d_n(x,y):=\max_{0\leq j\leq n-1}\limits d(f^{j}(x),f^j(y)).$$ Then  \emph{Bowen open  ball } of radius $\epsilon$ and  order $n$ in the metric $d_n$ around $x$ is   given by
$$B_n(x,\epsilon)=\{y\in X: d_n(x,y)<\epsilon\}.$$
Therefore, we say that $E$ is an $(n,\epsilon)$ spanning set of $Z\subset X$ if for any $y\in Z$ there exist $x\in E$ such that $d_n(x,y)\leq\epsilon$ and $F$ is an $(n,\epsilon)$-separated set of $Z\subset X$ if for any $x\neq y\in F$, $d_n(x,y)>\epsilon.$
Denote $\#A$ to be the cardinality of set $A$, $\partial A$ to be the boundary of set $A$ and $C(X,\mathbb{R})$ to be the set of all  continuous functions.
Denote $\mathcal{M}_f(X)$ to be the set that is consisted of all $f$-invariant Borel probability measures on $X$ and denote $\mathcal{M}_f^e(X)$ to be the set that is consisted of all $f$-ergodic invariant Borel probability measures on $X.$ 
Let $S_n\psi(x):=\sum_{i=0}^{n-1}\psi(f^ix)$ and 
$Var(\varphi,\epsilon):=\max\{|\varphi(x)-\varphi(y)|:d(x,y)<\epsilon\}$.

\subsection{Specification Property}
Let $(X,d,f)$ be a $\mathrm{TDS}$.   $f$ satisfies the  specification property means that for every $\epsilon>0$, there exists an integer $m=m(\epsilon)$ such that for any finite integer intervals $\{[a_j,b_j]_{j=1}^k\}$ with $a_{j+1}-b_j\ge m$ for $j\in \{1,\dots,k-1\}$ and any $x_1,\dots,x_k$ in $X$, there exists a point $x\in X$ such that
$$d_{b_j-a_j}(f^{a_j}x,x_j)<\epsilon \text{ for  all } \ j=1,\dots,k.$$

\subsection{Bowen metric mean dimension with potential for  subsets}

Given a set $Z\subset X, N\in \mathbb N, 0<\epsilon<1,s\in \mathbb R,$ and a potential $\psi\in C(X,\mathbb{R})$, we consider 
$$m_{N,\epsilon}(f,Z,s,d,\psi)=\inf\left\{\sum_{i\in I}\exp\left(-sn_i+S_{n_i}\psi(x_i)\cdot\Big(\log\frac{1}{\epsilon}\Big)\right)\right\},$$
where the infimum is taken over all finite or countable covers $\{B_{n_i}(x_i,\epsilon)\}_{i\in I}$ of $Z$ with $n_i \geq N$. Obviously, the limit 
$$m_{\epsilon}(f,Z,s,d,\psi)=\lim\limits_{N\to \infty}m_{N,\epsilon}(f,Z,s,d,\psi)$$
exists since $m_{N,\epsilon}(f,Z,s,d,\psi)$ is non-increasing when $N$ increases.
 $m_{\epsilon}(f,Z,s,d,\psi)$ has a critical value of parameter $s$ jumping from $\infty$ to $0$ and which is defined by
\begin{align*}
M_{\epsilon}(f,Z,d,\psi)&:=\inf\{s:m_{\epsilon}(f,Z,s,d,\psi)=0\}\\
&:=\sup\{s:m_{\epsilon}(f,Z,s,d,\psi)=\infty\}.
\end{align*}
The Bowen upper metric mean dimension of $f$ on $Z$ with potential $\psi$ is then defined as the following limit.
$$\overline{\rm{midm}}_M^B(f,Z,d,\psi)=\limsup\limits_{\epsilon \to 0}\frac{M_{\epsilon}(f,Z,d,\psi)}{\log\frac{1}{\epsilon}}.$$
Similarly, the  Bowen lower metric mean dimension of $f$ on $Z$ with potential $\psi$ is defined as 
$$\underline{\rm{midm}}_M^B(f,Z,d,\psi)=\liminf\limits_{\epsilon \to 0}\frac{M_{\epsilon}(f,Z,d,\psi)}{\log\frac{1}{\epsilon}}.$$
When  $\psi=0$ and   $Z=X$, $\overline{\rm{midm}}_M^B(f,X,d,\psi)$ is equal to the usual metric mean dimension  introduced by Lindenstrauss and Weiss \cite{lw00}.

\subsection{Metric mean dimension with potential for subsets}
Given a set $Z\subset X, n\in \mathbb N, 0<\epsilon<1$ and  $\psi\in C(X,\mathbb{R})$. We set
\begin{align*}
	&Q_n(Z,\psi,\epsilon):=\inf\left\{\sum_{x\in E}^{}\exp\left\{\left(\log\dfrac{1}{\epsilon} \right) S_n\psi(x)\right\}:E\ \text{is an}\ (n,\epsilon)\ \text{spanning set for}\ Z\right\},\\
	&P_n(Z,\psi,\epsilon):=\sup\left\{\sum_{x\in E}^{}\exp\left\{\left(\log\dfrac{1}{\epsilon} \right)S_n\psi(x)\right\}:E \ \text{is an}\ (n,\epsilon)\ \text{separated set for}\ Z\right\}.\\
\end{align*}	
The upper(lower) metric mean dimension of $f$ on $Z$ with potential $\psi$ are then defined as the following limits.
$$\overline{\rm{midm}}_M(f,Z,d,\psi):=\limsup\limits_{\epsilon \to 0}\frac{1}{\log\frac{1}{\epsilon}}\limsup_{n\to\infty}\frac{\log Q_n(Z,\psi,\epsilon)}{n},$$
$$\underline{\rm{midm}}_M(f,Z,d,\psi):=\liminf\limits_{\epsilon \to 0}\frac{1}{\log\frac{1}{\epsilon}}\limsup_{n\to\infty}\frac{\log Q_n(Z,\psi,\epsilon)}{n}.$$
The upper(lower) metric mean dimension of $f$ on $Z$ with   $\psi$  can also be defined via separated sets by replacing $Q_n$ by $P_n$ \cite{ycz22}.
Now, we give some properties for  upper (lower) metric mean dimension of $f$ on $Z$ with potential $\psi$.
\begin{prop}\label{prop A}
	For  $\psi\in C(X,\mathbb{R})$, $c\in \mathbb R$ and $Z\subset X$, we have 
	$$\overline{\rm{midm}}_M(f,Z,d,\psi+c)=\overline{\rm{midm}}_M(f,Z,d,\psi)+c.$$
\end{prop}
\begin{proof}
Obviously, $S_n(\psi+c)(x)=S_n\psi(x)+nc$. Then, $Q_n(Z,\psi+c,\epsilon)=Q_n(Z,\psi,\epsilon)\cdot\exp\{nc(\log\frac{1}{\epsilon})\}$. Immediately, $\overline{\rm{midm}}_M(f,Z,d,\psi+c)=\overline{\rm{midm}}_M(f,Z,d,\psi)+c.$
\end{proof}
For any $\psi\in C(X,\mathbb{R})$, by the compactness of $X$, there exists a constant $c$ such that $\psi+c>0$. Combing Proposition \ref{prop A} and \cite[Proposition 2.1]{zcz24}, we can assume  that  $\psi$ is a positive function and $\overline{\rm{midm}}_M(f,Z,d,\psi)$, $\overline{\rm{midm}}_M^B(f,Z,d,\psi)>0$.

Let $\mu\in\mathcal{M}_f(X)$, $\epsilon>0$, $n\in\mathbb{N}$ and $\delta>0$. Put $N_{n}^{\mu}(\delta,\epsilon):=\min\{\#E:E\subset X\ \text{and}\ \mu(\bigcup_{x\in E}B_n(x,\epsilon))>1-\delta\}.$ Define the upper and lower Katok's entropy of $\mu$ as 
\begin{align*}
	\overline{h}_{\mu}(f,\delta,\epsilon)&=\limsup_{n\to\infty}\frac{1}{n}\log N_{n}^{\mu}(\delta,\epsilon),\\
	\underline{h}_{\mu}(f,\delta,\epsilon)&=\liminf_{n\to\infty}\frac{1}{n}\log N_{n}^{\mu}(\delta,\epsilon).
\end{align*}
\begin{lem}\label{lem 2.2}\cite{shi22}
	Let $\mu\in \mathcal{M}_f^e(X), 0<\epsilon_2<\epsilon_1$ and $\mathcal{U}$ be a finite open over of $X$ with $diam(\mathcal{U})\leq\epsilon_1$ and $Leb(\mathcal{U})\ge\epsilon_2$. Then for any $\delta\in(0,1)$, we have
	$$\overline{h}_{\mu}(f,\delta,\epsilon_1)\leq \inf_{\xi\succ\mathcal{U}}h_{\mu}(f,\xi)\leq\overline{h}_{\mu}(f,\delta,\epsilon_2),$$
	$$\underline{h}_{\mu}(f,\delta,\epsilon_1)\leq \inf_{\xi\succ\mathcal{U}}h_{\mu}(f,\xi)\leq\underline{h}_{\mu}(f,\delta,\epsilon_2).$$
	\end{lem}

\begin{prop}\label{prop 2.1}\cite[Proposition 1.3]{cl23}
	Let $(X,f)$ be a $\mathrm{TDS}$ and $\mu\in\mathcal{M}_f^e(X)$. For $\epsilon>0,\ \delta\in(0,1)$ and $\psi \in C(X,\mathbb{R})$, we have
	$$\lim\limits_{\delta\to0}\limsup_{n\to\infty}\frac{1}{n}\log{N_n^{\mu}\left(\psi,\frac{\delta}{2},\epsilon\right)}\ge\lim\limits_{\delta \to0}\overline{h}_{\mu}(f,\delta,\epsilon)+\Big(\log\frac{1}{4\epsilon}\Big)\int\psi \mathrm{d}{\mu},$$
	$$\lim\limits_{\delta\to0}\liminf_{n\to\infty}\frac{1}{n}\log{N_n^{\mu}\left(\psi,\frac{\delta}{2},\epsilon\right)}\ge\lim\limits_{\delta \to0}\underline{h}_{\mu}(f,\delta,\epsilon)+\Big(\log\frac{1}{4\epsilon}\Big)\int\psi \mathrm{d}{\mu}.$$
where
\begin{align*}
    &N_n^{\mu}(\psi,\delta,\epsilon):=\\
	&\inf\left\{\sum_{x\in E}^{}\exp\left\{\left(\log\dfrac{1}{\epsilon} \right)S_n\psi(x)\right\}:E \ \text{is an}\ (n,\epsilon)\ \text{spanning set of}\ G\subset X\ \text{with}\ \mu(G)\ge{1-\delta}\right\}\\
\end{align*}
\end{prop}
The following variational principles for the metric mean dimension with potential was obtained by Chen and Li.
\begin{lem}\label{lem 2.2}\cite[Theorem C]{cl23}
    Let $(X,d,f)$ be a $\mathrm{TDS}$ and $\epsilon>0$. For any $\delta\in(0,1)$ and $\psi\in C(X,\mathbb{R})$, we have that
    $$\overline{\rm{mdim}}_M(f,X,d,\psi)=\limsup_{\epsilon\to 0}\frac{1}{\log{\frac{1}{\epsilon}}}\sup_{\mu\in\mathcal{M}_f^e(X)}\Lambda(\epsilon,\delta,\psi),$$
    $$\underline{\rm{mdim}}_M(f,X,d,\psi)=\liminf_{\epsilon\to 0}\frac{1}{\log{\frac{1}{\epsilon}}}\sup_{\mu\in\mathcal{M}_f^e(X)}\Lambda(\epsilon,\delta,\psi).$$
where $\Lambda(\epsilon,\delta,\psi):=\limsup_{n\to\infty}\frac{1}{n}\log{N_n^{\mu}(\psi,\delta,\epsilon)}$.
\end{lem}

\begin{lem}\label{lem 2.3}\cite[Theorem B]{cl23}
    Let $(X,d,f)$ be a $\mathrm{TDS}$ and $\epsilon>0$. For any $\delta\in(0,1)$ and $\psi\in C(X,\mathbb{R})$, we have that
    $$\overline{\rm{mdim}}_M(f,X,d,\psi)=\limsup_{\epsilon\to 0}\frac{1}{\log{\frac{1}{\epsilon}}}\sup_{\mu\in\mathcal{M}_f^e(X)}\left\{\inf_{diam\xi<\epsilon}h_{\mu}(f,\xi)+\left(\log\dfrac{1}{\epsilon} \right)\int \psi \mathrm{d}{\mu}\right\},$$
    $$\underline{\rm{mdim}}_M(f,X,d,\psi)=\liminf_{\epsilon\to 0}\frac{1}{\log{\frac{1}{\epsilon}}}\sup_{\mu\in\mathcal{M}_f^e(X)}\left\{\inf_{diam\xi<\epsilon}h_{\mu}(f,\xi)+\left(\log\dfrac{1}{\epsilon} \right)\int \psi \mathrm{d}{\mu}\right\}.$$
\end{lem}

\section{Proof of main result} 
In this section, we give the proof of Theorem \ref{thm 1.1}. We now only consider the situation of upper metric mean dimension with potential. The result of lower metric mean dimension with potential can be proved in the similar way. We may assume that $I_{\varphi}\neq\emptyset$ and show that 
$$\overline{\rm{mdim}}_M^B(f,I_{\varphi},d,\psi)=\overline{\rm{mdim}}_M(f,X,d,\psi).$$
Since $I_{\varphi}\subset X$, we have
$$\overline{\rm{mdim}}_M^B(f,I_{\varphi},d,\psi)\leq\overline{\rm{mdim}}_M^B(f,X,d,\psi).$$
\cite[Proposition 3.4]{ycz22} shows that 
$$\overline{\rm{mdim}}_M(f,X,d,\psi)=\overline{\rm{mdim}}_M^B(f,X,d,\psi).$$
This implies $$\overline{\rm{mdim}}_M^B(f,I_{\varphi},d,\psi)\leq\overline{\rm{mdim}}_M(f,X,d,\psi).$$
Thus, we only need to show 
$$\overline{\rm{mdim}}_M^B(f,I_{\varphi},d,\psi)\ge\overline{\rm{mdim}}_M(f,X,d,\psi).$$
Let $S:=\overline{\rm{mdim}}_M(f,X,d,\psi)<\infty$.  Without losing generality, we only need to consider the case that $S>0$. 
\begin{lem}\label{lem 3.1}
   Given sufficiently small $\gamma\in(0,\min\{S/7,1\})$. There exists $\epsilon_0=\epsilon_0(\gamma)>0$ such that 
    \begin{align}
        \label{3.1}&\log{\frac{1}{5\epsilon_0}}>1,\\ 
		\label{3.2}&S-\frac{\gamma}{2}\leq\frac{1}{\log{\frac{1}{5\epsilon_0}}}\sup_{\mu\in\mathcal{M}_f^e(X)}\left\{\inf_{diam\xi<5\epsilon_0}h_{\mu}(f,\xi)+\Big(\log{\frac{1}{5\epsilon_0}}\Big)\int \psi \mathrm{d}{\mu}\right\},\\
		\label{3.3}&\sup_{\epsilon\in(0,5\epsilon_0)}\frac{M_{\epsilon}(f,I_{\varphi},d,\psi)}{\log{\frac{1}{\epsilon}}}\leq\overline{\rm{mdim}}_M^B(f,I_{\varphi},d,\psi)+\gamma.
\end{align}
Besides, there exist $\mu_1,\mu_2\in\mathcal{M}_f(X)$ such that
$$\int\varphi \mathrm{d}{\mu_1}\neq\int\varphi \mathrm{d}{\mu_2}$$
$$\frac{1}{\log{\frac{1}{5\epsilon_0}}}\inf_{diam\xi<5\epsilon_0}h_{\mu_i}(f,\xi)+\int \psi \mathrm{d}{\mu_i}>S-\gamma\ for\ i=1,2.$$
\end{lem}
\begin{proof}
    Firstly, we choose $\epsilon'>0$ sufficiently small such that for any $\epsilon\in(0,\epsilon')$, $\log{\frac{1}{5\epsilon}}>1$. Then, by Lemma \ref{lem 2.3}, Lemma \ref{lem 2.2} and definition of $\overline{\rm{mdim}}_M^B(f,I_{\varphi},d,\psi)$, we can choose $\epsilon_0=\epsilon_0(\gamma)\in(0,\epsilon')$ satisfying
\begin{align*}
    &S-\frac{\gamma}{2}\leq\frac{1}{\log{\frac{1}{5\epsilon_0}}}\sup_{\mu\in\mathcal{M}_f^e(X)}\left\{\inf_{diam\xi<5\epsilon_0}h_{\mu}(f,\xi)+\Big(\log{\frac{1}{5\epsilon_0}}\Big)\int \psi \mathrm{d}{\mu}\right\},\\
    &\sup_{\epsilon\in(0,5\epsilon_0)}\frac{M_{\epsilon}(f,I_{\varphi},d,\psi)}{\log{\frac{1}{\epsilon}}}\leq\overline{\rm{mdim}}_M^B(f,I_{\varphi},d,\psi)+\gamma.\\
\end{align*}
Choose $\mu_1\in\mathcal{M}_f^e(X)$ such that 
$$S-\frac{2\gamma}{3}\leq\frac{1}{\log{\frac{1}{5\epsilon_0}}}\inf_{diam\xi<5\epsilon_0}h_{\mu_1}(f,\xi)+\int \psi \mathrm{d}{\mu_1}.$$
Due to $I_{\varphi}\neq\emptyset$, there exists a ponit $x\in X$ such that
$$\lim\limits_{k\to\infty}\frac{1}{n_k}\sum_{i=0}^{n_k-1}\varphi(f^ix)=C\neq\int\varphi \mathrm{d}{\mu_1}.$$
Let $\nu_k:=\frac{1}{n_k}\sum_{i=0}^{n_k-1}\delta_{f^i(x)}\in\mathcal{M}(X),$ and  $\nu$ be some accumulation point of $\nu_k$, then $\nu\in\mathcal{M}_f(X)$ and $\int \varphi \mathrm{d}{\mu_1}\neq\int \varphi \mathrm{d}{\nu}$.
Choose $\mu_2=t\mu_1+(1-t)\nu$, where $t\in(0,1)$ is sufficiently close to 1 such that 
$$\frac{1}{\log{\frac{1}{5\epsilon_0}}}\inf_{diam\xi<5\epsilon_0}h_{\mu_2}(f,\xi)+\int \psi \mathrm{d}{\mu_2}>S-\gamma.$$
\end{proof}

In the following, we fix $\gamma,\ \epsilon_0>0$ and the $\mu_1,\mu_2$ be as Lemma \ref{lem 3.1}. Let $\alpha_i:=\int\varphi \mathrm{d}{\mu_i},$ for $i=1,2$.

\begin{rem}
    According to the construction of $\mu_2$, we can not guarantee that  $   \mu_2$ is   ergodic. For convenience, we construct an auxiliary measure to "appropriate" $\mu_2.$
\end{rem}
The following lemma is a generalized form of \cite[p.535]{ys90} and \cite[Lemma 7]{br23}.
\begin{lem}\label{lem 3.3}
	Suppose $\overline{\rm{mdim}}_M(f,X,d,\psi)<\infty$. Let $\mu\in\mathcal{M}_f(X)$, $\epsilon>0$, any $\delta>0$ and $\mathcal{U}$ be a finite open cover of $X$. There exists a measure $\nu\in\mathcal{M}_f(X)$ satisfying
	\begin{align*}
		&(1)\  \nu=\sum_{i=1}^{k}\lambda_i\nu_i,\  where\ \lambda_i>0,\ \sum_{i=1}^{k}\lambda_i=1\ and\ \nu_i\in\mathcal{M}_f^e(X);\\
		&(2)\  \inf_{\xi\succ\mathcal{U}}h_{\mu}(f,\xi)+\Big(\log{\frac{1}{\epsilon}}\Big)\int\psi \mathrm{d}{\mu}\leq\inf_{\xi\succ\mathcal{U}}h_{\nu}(f,\xi)+\Big(\log{\frac{1}{\epsilon}}\Big)\int\psi \mathrm{d}{\nu}+\delta;\\
		&(3)\  \left|\int\varphi \mathrm{d}{\nu}-\int\varphi \mathrm{d}{\mu}\right|<\delta.
	\end{align*}	
\end{lem}
\begin{proof}
	It is well known that the week*-topology on $\mathcal{M}(X)$ is metrizable, and denote $d_*$ to be one of the compatible metrics. Let $\beta>0$ be sufficiently small such that for every $\tau_1,\tau_2\in\mathcal{M}_f(X)$, if $d_*(\tau_1,\tau_2)<\beta,$ then
	$$\left|\int\varphi\mathrm{d}{\tau_1}-\int\varphi \mathrm{d}{\tau_2}\right|<\delta.$$ 
	Let $\mathcal{P}=\{P_1,\dots,P_{k}\}$ be a partition of $\mathcal{M}_f(X)$ whose diameter with respect to $d_*$ is smaller than $\beta.$ By  Ergodic Decomposition Theorem \cite[Remark (2)]{wal00}, there exists a measure $\hat{\mu}$ on $\mathcal{M}_f(X)$ satisfying $\hat{\mu}(\mathcal{M}_f^e(X))=1$ and
	$$\mu=\int_{\mathcal{M}_f^e(X)}^{}\tau \mathrm{d}{\hat{\mu}(\tau)},$$ i.e.,
	$$\int{\psi(x)\mathrm{d}{\mu(x)}}=\int_{\mathcal{M}_f^e(X)}^{}\left(\int\psi(x)d{\tau(x)}\right)\mathrm{d}{\hat{\mu}(\tau)}\ \text{for every}\ \psi\in C(X,\mathbb{R}).$$
	Set $\lambda_i=\hat{\mu}(P_i)$. Since
	$$\sup_{\tau\in\mathcal{M}_f^e(X)}\left\{\inf_{\xi\succ\mathcal{U}}h_{\tau}(f,\xi)+\Big(\log{\frac{1}{\epsilon}}\Big)\int\psi \mathrm{d}{\tau}\right\}<\infty,$$
	there exist $\nu_i\in P_i\cap{\mathcal{M}_f^e}(X)$ such that 
	$$\inf_{\xi\succ\mathcal{U}}h_{\nu_i}(f,\xi)+\Big(\log{\frac{1}{\epsilon}}\Big)\int\psi \mathrm{d}{\nu_i}\ge\inf_{\xi\succ\mathcal{U}}h_{\tau}(f,\xi)+\Big(\log{\frac{1}{\epsilon}}\Big)\int\psi \mathrm{d}{\tau}-\delta$$ 
	for $\hat{\mu}-$almost every $\tau\in P_i\cap{\mathcal{M}_f^e}(X).$ 
Let $\nu:=\sum_{i=1}^{k}\lambda_i\nu_i$. Clearly, $\nu$ satisfies (1) and (3). For (2), by \cite[Proposition 5]{hmry04}, one has 
	$$\inf_{\xi\succ\mathcal{U}}h_{\mu}(f,\xi)=\int_{\mathcal{M}_f^e(X)}^{}\inf_{\xi\succ\mathcal{U}}h_{\tau}(f,\xi)\mathrm{d}{\hat{\mu}(\tau)}.$$
	Thus, we have 
	\begin{align*}
	&\inf_{\xi\succ\mathcal{U}}h_{\mu}(f,\xi)+\Big(\log{\frac{1}{\epsilon}}\Big)\int\psi \mathrm{d}{\mu}\\
	=&\sum_{i=1}^{k}\int_{P_i\cap\mathcal{M}_f^e(X)}^{}\left\{\inf_{\xi\succ\mathcal{U}}h_{\tau}(f,\xi)+\left(\int\psi(x)\mathrm{d}{\tau(x)}\right)\log{\frac{1}{\epsilon}}\right\}\mathrm{d}{\hat{\mu}(\tau)}\\
	\leq&\sum_{i=1}^{k}\lambda_i\left\{\inf_{\xi\succ\mathcal{U}}h_{\nu_i}(f,\xi)+\left(\int\psi(x)\mathrm{d}{\nu_i}\right)\log{\frac{1}{\epsilon}}\right\}+\delta\\
	\leq&\inf_{\xi\succ\mathcal{U}}h_{\nu}(f,\xi)+\Big(\log{\frac{1}{\epsilon}}\Big)\int\psi \mathrm{d}{\nu}+\delta.
\end{align*}
\end{proof}

\subsection{Construction of the Moran-like Fractal}
The construction of the Moran-like Fractal is a standard process that inspired by Thompson \cite{th10} and Backes \cite{br23}. By \cite[Lemma 6]{shi22}, there exists  a finite open cover $\mathcal{U}$ of $X$  that satisfies
    $$\text{diam}(\mathcal{U})\leq5\epsilon_0\ ,\ \text{Leb}(\mathcal{U})\ge\frac{5\epsilon_0}{4}.$$
	Fix $\delta\in(0,\gamma/2)$. By Lemma \ref{lem 3.3}, there exists $\nu\in\mathcal{M}_f(X)$ satisfying
\begin{align*}
    &(1)\  \nu=\sum_{i=1}^{k}\lambda_i\nu_i,\  where\ \lambda_i>0,\ \sum_{i=1}^{k}\lambda_i=1\ and\ \nu_i\in\mathcal{M}_f^e(X);\\
    &(2)\   \inf_{\xi\succ\mathcal{U}}h_{\mu_1}(f,\xi)+\Big(\log{\frac{1}{\epsilon}}\Big)\int\psi \mathrm{d}{\mu_1}\leq\inf_{\xi\succ\mathcal{U}}h_{\nu}(f,\xi)+\Big(\log{\frac{1}{\epsilon}}\Big)\int\psi \mathrm{d}{\nu}+\delta;\\
    &(3)\  \left|\int\varphi \mathrm{d}{\nu}-\int\varphi \mathrm{d}{\mu_1}\right|<\delta.
\end{align*} 
Since $\nu_i\in\mathcal{M}_f^e(X)$, there exist $N_0\in\mathbb{N}$ large enough such that the set
$$Y_i(N_0)=\left\{x\in X:\left|\frac{1}{n}S_n\varphi(x)-\int\varphi d{\nu_i}\right|<\delta,  \forall\ n\ge N_0\right\}$$
has $\nu_i-$measure at least $1-\gamma$ for every $i\in\{1,\dots,k\}.$

\begin{lem}\label{lem 3.4}
	For  $\epsilon_0$ and $\delta\in(0,1)$, we can find a large enough $\hat{n}$ and  $\mathcal{S}_{1,i}$ such that $\mathcal{S}_{1,i}$ is an $\left([\lambda_i\hat{n}],\frac{5\epsilon_0}{4}\right)$ separated set for $Y_i(N_0)$ and 
	\begin{align*}
	M_{1,i}:&=\sum_{x\in\mathcal{S}_{1,i}}^{}\exp(S_{[\lambda_i\hat{n}]}\psi(x)\cdot\log{\frac{1}{5\epsilon_0}})\\
	&\ge\exp\left\{[\lambda_i\hat{n}]\left(\inf_{\xi\succ\mathcal{U}}h_{\nu_i}(f,\xi)+\Big(\log\frac{1}{5\epsilon_0}\Big)\int\psi d{\nu_i}-2\gamma\right)\right\}.
	\end{align*}
	Furthermore,  $[\lambda_i\hat{n}]$ can be chosen such that $[\lambda_i\hat{n}]\ge N_0$ and $\hat{n}\ge 2^{m}$ where $m=m(\epsilon/16)$ is as in the definition of the specification property and $\epsilon$ will be determined later.
\end{lem}

\begin{proof}
	By Proposition \ref{prop 2.1}, we have  
	\begin{align*}
	\liminf_{n\to\infty}\frac{1}{n}\log{N_n^{\nu_i}}\left(\psi,\delta,\frac{5\epsilon_0}{4}\right)&\ge \underline{h}_{\nu_i}\left(f,\delta,\frac{5\epsilon_0}{4}\right)+\Big(\log\frac{1}{5\epsilon_0}\Big)\int\psi d{\nu_i}-\frac{\gamma}{2}\\
&\ge \inf_{\xi\succ\mathcal{U}}h_{\nu_i}(f,\xi)+\Big(\log\frac{1}{5\epsilon_0}\Big)\int\psi d{\nu_i}-\frac{\gamma}{2}.
\end{align*}
Since $Q_n(Z,\psi,\epsilon)\leq P_n(Z,\psi,\epsilon)$ and $\nu_i(Y_i(N_0))>1-\gamma$, immediately, we have that
$$Q_n\left(Y_i(N_0),\psi,\frac{5\epsilon_0}{4}\right)\ge N_n^{\nu_i}\left(\psi,\delta,\frac{5\epsilon_0}{4}\right).$$
Let $M(1,n)=P_n(Y_i(N_0),\psi,\frac{5\epsilon_0}{4})$. Then we obtain 
\begin{align*}
   \liminf_{n\to\infty}\frac{1}{n}\log{M(1,n)}&\ge\liminf_{n\to\infty}\frac{1}{n}\log{N_n^{\nu_i}\left(\psi,\delta,\frac{5\epsilon_0}{4}\right)}\\
   &\ge \underline{h}_{\nu_i}\left(f,\frac{5\epsilon_0}{4},\delta\right)+\Big(\log\frac{1}{5\epsilon_0}\Big)\int\psi d{\nu_i}-\frac{\gamma}{2}.
\end{align*}
Thus, we can choose a sequence $[\lambda_i\hat{n}]\to\infty$ as $\hat{n}\to\infty$ so that 
$$\frac{1}{[\lambda_i\hat{n}]}\log{M(1,n)}\ge\underline{h}_{\nu_i}\left(f,\frac{5\epsilon_0}{4},\delta\right)+\Big(\log\frac{1}{5\epsilon_0}\Big)\int\psi d{\nu_i}-\gamma.$$
Taking $\mathcal{S}_i$ be a choice of $\left([\lambda_i\hat{n}],\frac{5\epsilon_0}{4}\right)$ separated set for $Y_i(N_0)$ that satisfies
$$\frac{1}{[\lambda_i\hat{n}]}\log{\sum_{x\in\mathcal{S}_i}^{}\exp\left(S_{[\lambda_i\hat{n}]}(x)\log{\frac{1}{5\epsilon_0}}\right)}\ge \frac{1}{[\lambda_i\hat{n}]}\log{M(1,n)}-\gamma.$$
Then, we have that
\begin{align*}
	&\frac{1}{[\lambda_i\hat{n}]}\log{M_{1,i}}\ge \underline{h}_{\nu_i}\left(f,\frac{5\epsilon_0}{4},\delta\right)+\Big(\log\frac{1}{5\epsilon_0}\Big)\int\psi d{\nu_i}-2\gamma,\\
	&M_{1,i}\ge\exp\left\{[\lambda_i\hat{n}]\left(\inf_{\xi\succ\mathcal{U}}h_{\nu_i}(f,\xi)+\Big(\log\frac{1}{5\epsilon_0}\Big)\int\psi d{\nu_i^k}-2\gamma\right)\right\}.
\end{align*}
\end{proof}

We choose $\epsilon\in(0,\epsilon_0)$ satisfying $Var(\psi,\epsilon)<\gamma$, and fix all the ingredients provided by Lemma \ref{lem 3.4}. For every $y_i\in \mathcal{S}_{i}$, by the specification property, there exists $x=x(y_1,\dots,y_{k})\in X$ that satisfies
$$d_{[\lambda_l\hat{n}]}(y_l,f^{a_l}x)<\frac{\epsilon}{16}$$
for $l\in\{1,\dots,k\}$, where $a_1=0$ and $a_l=\sum_{i=1}^{l-1}[\lambda_i\hat{n}]+(l-1)m$ for $l\in\{2,\dots,k\}.$
Let $\mathcal{S}_1$ be the set that consist of such $x=x(y_1,\dots,y_k)$. Define $n_1=\sum_{i=1}^{k}[\lambda_i\hat{n}]+(k-1)m$, and we have that $\frac{n_1}{\hat{n}}\to1$ as $\hat{n}\to\infty$. 
We claim that $\mathcal{S}_1$ is a $(n_1,9\epsilon_0/8)$-separated set and if $(y_1,\dots,y_k)\neq(y_1',\dots,y_k')$ then $x=x(y_1,\dots,y_k)\neq x'=x'(y_1',\dots,y_k')$. Actually, if $y_l\neq y_l'$ for some $l\in \{1,\dots,k\}$, we have
\begin{align*}
	\frac{5\epsilon_0}{4}&<d_{[\lambda_l\hat{n}]}(y_l,y_l')\\
	&\leq d_{[\lambda_l\hat{n}]}(y_l,f^{a_l}x)+d_{[\lambda_l\hat{n}]}(y_l',f^{a_l}x')+d_{[\lambda_l\hat{n}]}(f^{a_l}x,f^{a_l}x')\\
	&<2\frac{\epsilon}{16}+d_{[\lambda_l\hat{n}_k]}(f^{a_l}x,f^{a_l}x').
\end{align*}
Thus,
$$d_{n_k}(x,x')\ge d_{[\lambda_l\hat{n}]}(f^{a_l}x,f^{a_l}x')>\frac{5\epsilon_0}{4}-\frac{\epsilon}{8}\ge\frac{9\epsilon_0}{8}.$$
We define $M_1:=M_{1,1}\cdots M_{1,k}.$

\begin{lem}\label{lem 3.5}
	For $\hat{n}$ sufficiently large such that 
	$$k\leq\frac{\delta\hat{n}}{4\|\varphi\|+1}\ \text{and}\  \frac{2(k-1)m\|\varphi\|}{n}<\frac{\delta}{2},$$
	we have 
	\begin{align*}
		&(1)\ M_1\ge\exp\left\{n_1(\inf_{diam\xi<5\epsilon_0}h_{\mu_1}(f,\xi)+\Big(\log{\frac{1}{5\epsilon_0}}\Big)\int\psi d{\mu_1}-5\gamma)\right\},\\
		&(2)\ if\  x\in\mathcal{S}_1,\ \left|\frac{1}{n_1}S_{n_1}\varphi(x)-\alpha_1\right|< 4\delta.
	\end{align*}
\end{lem}
\begin{proof}
	(1): For sufficiently large $\hat{n}$, according to Lemma \ref{lem 3.4}, we have that
	\begin{align*}
		M_1&\ge\exp\left\{\sum_{i=1}^{k}[\lambda_i\hat{n}]\left(\inf_{\xi\succ\mathcal{U}}h_{\nu_i}(f,\xi)+\Big(\log{\frac{1}{5\epsilon_0}}\Big)\int\psi \mathrm{d}{\nu_i}-2\gamma\right)\right\}\\
		&=\exp\left\{\sum_{i=1}^{k}\frac{[\lambda_i\hat{n}]}{\lambda_i}\left(\lambda_i\inf_{\xi\succ\mathcal{U}}h_{\nu_i}(f,\xi)+\lambda_i\Big(\log{\frac{1}{5\epsilon_0}}\Big)\int\psi \mathrm{d}{\nu_i}-2\lambda_i\gamma\right)\right\}\\
		&\ge\exp\left\{\sum_{i=1}^{k}\hat{n}\left(\lambda_i\inf_{\xi\succ\mathcal{U}}h_{\nu_i}(f,\xi)+\lambda_i\Big(\log{\frac{1}{5\epsilon_0}}\Big)\int\psi \mathrm{d}{\nu_i}-3\lambda_i\gamma\right)\right\}\\
		&\ge\exp\left\{\hat{n}\left(\inf_{\xi\succ\mathcal{U}}h_{\mu_1}(f,\xi)+\Big(\log{\frac{1}{5\epsilon_0}}\Big)\int\psi \mathrm{d}{\mu}-3\gamma-\delta\right)\right\}\\
		&\ge\exp\left\{\hat{n}\left(\inf_{\xi\succ\mathcal{U}}h_{\mu_1}(f,\xi)+\Big(\log{\frac{1}{5\epsilon_0}}\Big)\int\psi \mathrm{d}{\mu}-\frac{7}{2}\gamma\right)\right\}\\
		&=\exp\left\{\frac{\hat{n}}{n_1}\left(n_1\left(\inf_{\xi\succ\mathcal{U}}h_{\mu}(f,\xi)+\Big(\log{\frac{1}{5\epsilon_0}}\Big)\int\psi \mathrm{d}{\mu}-\frac{7}{2}\gamma\right)\right)\right\}\\
		&\ge\exp\left\{{n_1}\left(\inf_{\xi\succ\mathcal{U}}h_{\mu_1}(f,\xi)+\Big(\log{\frac{1}{5\epsilon_0}}\Big)\int\psi \mathrm{d}{\mu_1}-4\gamma\right)\right\}\\
		&\ge\exp\left\{{n_1}\left(\inf_{diam\xi<5\epsilon_0}h_{\mu_1}(f,\xi)+\Big(\log{\frac{1}{5\epsilon_0}}\Big)\int\psi \mathrm{d}{\mu_1}-4\gamma\right)\right\}.\\
	\end{align*}
	\begin{align*}
		(2):&\left|S_{n_1}\varphi(x)-n_1\alpha_1\right|\\
		&\leq\sum_{i=1}^{k}\left(\left|S_{[\lambda_i\hat{n}]}\varphi(f^{a_i}x)-S_{[\lambda_i\hat{n}]}\varphi(y_i)\right|+\left|S_{[\lambda_i\hat{n}]}\varphi(y_i)-[\lambda_i\hat{n}]\int\varphi \mathrm{d}{\nu_i}\right|\right)\\
		&+\left|\sum_{i=1}^{k}\left([\lambda_i\hat{n}]\int\varphi \mathrm{d}{\nu_i}\right)-\left(\sum_{i=1}^{k}[\lambda_i\hat{n}]\right)\alpha_1\right|+2(k-1)m\|\varphi\|,\\
	\end{align*}
	where we use the fact that $|\alpha_1|\leq\|\varphi\|.$ We have 
	\begin{align*}
		&\sum_{i=1}^{k}\left|S_{[\lambda_i\hat{n}]}\varphi(f^{a_i}x)-S_{[\lambda_i\hat{n}]}\varphi(y_i)\right|\leq\sum_{i=1}^{k}[\lambda_i\hat{n}]Var(\varphi,\frac{\epsilon}{16})\leq\sum_{i=1}^{k}[\lambda_i\hat{n}]\delta,\\
		&\sum_{i=1}^{k}\left|S_{[\lambda_i\hat{n}]}\varphi(y_i)-[\lambda_i\hat{n}]\int\varphi \mathrm{d}{\nu_i}\right|\leq\sum_{i=1}^{k}[\lambda_i\hat{n}]\delta,\\
	\end{align*}
	\begin{align*}
		&\left|\sum_{i=1}^{k}\left([\lambda_i\hat{n}]\int\varphi \mathrm{d}{\nu_i}\right)-\left(\sum_{i=1}^{k}[\lambda_i\hat{n}]\right)\alpha\right|\\
		&\leq\left|\hat{n}\int\varphi\mathrm{d}{\nu}-\hat{n}\alpha_1\right|+2(\hat{n}-\sum_{i=1}^{k}[\lambda_i\hat{n}])\|\varphi\|\leq \hat{n}\delta+2k\|\varphi\|\\
		&\leq \hat{n}\delta+\frac{\delta\hat{n}}{2}\leq \frac{3\delta\hat{n}}{2}.
	\end{align*}
	Thus, 
	\begin{align*}
		&2\sum_{i=1}^{k}[\lambda_i\hat{n}]\delta+2(k-1)m\|\varphi\|+\frac{3\delta\hat{n}}{2}\\
&\leq n_1(2\delta+\frac{3\delta}{2})+\frac{\delta n_1}{2}=4n_1\delta.
	\end{align*}
	Combining these inequalities , we have that $\left|\frac{1}{n_1}S_{n_1}\varphi(x)-\alpha_1\right|< 4\delta.$
\end{proof}

Now, we begin to construct the Moran-like fractal. Let $\rho:\mathbb{N}\to\{1,2\}$ and $\rho(k)=(k+1(\text{mod}\ 2))+1.$ Then let $\{\delta_k\}_{k=1}^{\infty}$ be a strictly decreasing sequence with the property that $\delta_k\to0$ as $k\to\infty$ and $\delta_1<\gamma/2.$
Let $m_k=m(\epsilon_0/2^{k+5})$ be the gap in the definition of specification property. Then, for each $k\in\mathbb{N}$, resembling the Lemma \ref{lem 3.5}, exists $n_k>2^{m_k}$  and the set $\mathcal{S}_k$ such that 
$$M_k\ge\exp\left\{n_k(\inf_{diam\xi<5\epsilon_0}h_{\mu}(f,\xi)+\Big(\log{\frac{1}{5\epsilon_0}}\Big)\int\psi\mathrm{d}{\mu_{\rho(k)}}-5\gamma)\right\}$$
and for $x\in\mathcal{S}_k$
$$\left|\frac{1}{n_k}S_{n_k}\varphi(x)-\alpha_{\rho(k)}\right|< 4\delta_k.$$
We choose a sequence of positive integers $\{N_k\}_{k\in\mathbb{N}}$ such that $N_1=1$ and
$$\lim\limits_{k\to\infty}\frac{n_{k+1}+m_{k+1}}{N_k}=0,\ \lim\limits_{k\to\infty}\frac{N_1(n_1+m_1)+\dots+N_k(n_{k+1}+m_{k+1})}{N_{k+1}}=0.$$

$\mathbf{Step\ 1.}$ Constructions of intermediate sets $\{\mathcal{C}_k\}_{k=1}^{\infty}.$

For every $k\in\N$ and $\mathcal{S}_k:=\{x_i^k:i=1,\dots,\#\mathcal{S}_k\}$, we consider $\underline{i}=(i_1,\dots,i_{N_k})\in\{1,\dots,\#\mathcal{S}_k\}^{N_k}$. By the specification property, there exists a point $y:=y(i_1,\dots,i_{N_k})$ which satisfies
$$d_{n_k}(x_{i_j}^k,f^{a_j}y)<\frac{\epsilon_0}{2^{k+5}},\ \text{for}\ j\in\{1,\dots,N_k\},\ a_j=(j-1)(n_k+m_k).$$
We collect all the shadowing points into
$$\mathcal{C}_k=\{y(i_1,\dots,i_{N_k})\in X:(i_1,\dots,i_{N_k})\in\{1,\dots,\#\mathcal{S}_k\}^{N_k}\}.$$
Denote the amount of time for which the orbit of points in $\mathcal{C}_k$ has been shadowed by $c_k=N_kn_k+(N_k-1)m_k$ and we have the following lemma.
\begin{lem}\label{lem 3.6}
	Let $\underline{i},\ \underline{j}$ be the different words in $\{1,\dots,\#\mathcal{S}_k\}^{N_k}$. Then $y_1:=y(\underline{i})$ and $y_2:=y(\underline{j})$ are $(c_k,\frac{17\epsilon_0}{16})$-separated points, i.e. $d_{c_k}(y_1,y_2)>{\frac{17\epsilon_0}{16}}$. As a consequence, $\#\mathcal{C}_k=(\#\mathcal{S}_k)^{N_k}.$
\end{lem}
\begin{proof}
	Since $\underline{i}\neq\underline{j}$, there exists $l\in\{1,\dots,N_k\}$ such that $i_l\neq j_l.$ We have 
	\begin{align*}
		d_{c_k}(y_1,y_2)&\ge d_{n_k}(f^{a_l}y_1,f^{a_l}y_2)\\
&\ge d_{n_k}(x_{i_l}^k,x_{j_l}^k)-d_{n_k}(f^{a_l}y_1,x_{i_l}^k)-d_{n_k}(f^{a_l}y_2,x_{j_l}^k)\\
&>\frac{9\epsilon_0}{8}-\frac{\epsilon_0}{2^{(k+5)}}-\frac{\epsilon_0}{2^{(k+5)}}\\
&\ge\frac{17\epsilon_0}{16}.
	\end{align*}
\end{proof}

$\mathbf{Step\ 2.}$ Constructions of $\{\mathcal{T}_k\}_{k=1}^{\infty}$, the $k$-th level of the Moran-like fractal.

We define $\mathcal{T}_k$ inductively. Let $\mathcal{T}_1=\mathcal{C}_1$ and $t_1=c_1$. We construct $\mathcal{T}_{k+1}$ from $\mathcal{T}_k$ as follows. Let $t_{k+1}:=t_k+m_{k+1}+c_{k+1}$ and $x\in\mathcal{T}_k,\ y\in\mathcal{C}_{k+1}$. By the specification property, we can find a point $z:=z(x,y)$ that satisfies
$$d_{t_k}(x,z)<\frac{\epsilon_0}{2^{(k+6)}}\ and\ d_{c_{k+1}}(y,f^{t_k+m_{k+1}}z)<\frac{\epsilon_0}{2^{(k+6)}}.$$ 
Define $\mathcal{T}_{k+1}=\{z(x,y):x\in\mathcal{T}_k,\ y\in\mathcal{C}_{k+1}\}$, and note that $t_{k+1}$ is the amount of time for which the orbits of points in $\mathcal{T}_k$ has been shadowed. Similarly, we have the following lemma.
\begin{lem}\label{lem 3.7}
	For every $x\in\mathcal{T}_k$ and distinct points $y_1,y_2\in\mathcal{C}_{k+1}$,
	$$d_{t_k}(z(x,y_1),z(x,y_2))<\frac{\epsilon_0}{2^{k+5}},\ d_{t_{k+1}}(z(x,y_1),z(x,y_2))>\frac{33\epsilon_0}{32}.$$
	Thus, $\mathcal{T}_k$ is a $\left(t_k,\frac{33\epsilon_0}{32}\right)$-separated set. In particular, if $z_1,z_2\in\mathcal{T}_k$, then 
	$$\overline{B}_{t_k}\left(z_1,\frac{\epsilon_0}{2^{k+5}}\right)\cap\overline{B}_{t_k}\left(z_2,\frac{\epsilon_0}{2^{k+5}}\right)=\emptyset.$$
\end{lem}
\begin{proof}
	Let $z_1=z(x,y_1),\ z_2=(x,y_2)$. Hence, we have 
	$$d_{t_k}(z_1,z_2)\leq d_{t_k}(z_1,x)+d_{t_k}(z_2,x)<\frac{\epsilon_0}{2^{(k+6)}}+\frac{\epsilon_0}{2^{(k+6)}}=\frac{\epsilon_0}{2^{(k+5)}},$$
	\begin{align*}
		d_{t_{k+1}}(z_1,z_1)&\ge d_{c_{k+1}}(f^{t_k+m_{k+1}}z_1,f^{t_k+m_{k+1}}z_2)\\
		&>\frac{17\epsilon_0}{16}-\frac{\epsilon_0}{2^{(k+6)}}-\frac{\epsilon_0}{2^{(k+6)}}\ge\frac{33\epsilon_0}{32}.
	\end{align*}
	The third statement is a straight  consequence of the second inequality.
\end{proof}
As a direct result of the Lemma \ref{lem 3.7}, we have 
$$\#\mathcal{T}_k=\#\mathcal{T}_{k-1}\#\mathcal{C}_k=\#\mathcal{C}_1\dots\#\mathcal{C}_k=\#{\mathcal{S}_1}^{N_1}\dots\#{\mathcal{S}_k}^{N_k}.$$
\begin{lem}\label{lem 3.8}
	Let $z=z(x,y)\in\mathcal{T}_k$, then we have
	$$\overline{B}_{t_{k+1}}\left(z,\frac{\epsilon_0}{2^{k+6}}\right)\subset\overline{B}_{t_k}\left(x,\frac{\epsilon_0}{2^{k+5}}\right).$$
\end{lem}
\begin{proof}
	From the construction of $\mathcal{T}_k$, $d_{t_k}(z,x)<\frac{\epsilon_0}{2^{k+6}}.$ Thus, for any point $p\in\overline{B}_{t_{k+1}}(z,\frac{\epsilon_0}{2^{k+6}})$, one has
	$$d_{t_k}(p,x)\leq d_{t_k}(p,z)+d_{t_k}(z,x)\leq \frac{\epsilon_0}{2^{k+6}}\cdot2\leq\frac{\epsilon_0}{2^{k+5}},$$
	that implies  $p\in\overline{B}_{t_k}(x,\frac{\epsilon_0}{2^{k+5}})$. Therefore, the result has been proved.
\end{proof}

$\mathbf{Step\ 3.}$ Constructions of the Moran-like fractal and show that it is contained in $I_{\varphi}.$

Let $F_k=\cup_{x\in\mathcal{T}_k}\overline{B}_{t_k}\left(x,\frac{\epsilon_0}{2^{k+5}}\right)$. By Lemma \ref{lem 3.8}, $F_{k+1}\subset F_k$ and we have a decreasing sequence of compact sets, the set $F=\cap_kF_k$ is non-empty. Besides, every point $p\in F$ can be uniquely represented by a sequence $\underline{p}=(\underline{p}_1,\underline{p}_2,\dots)$, where each $\underline{p}_i=(\underline{p}_1^i,\dots,\underline{p}_{N_i}^i)\in\{1,2,\dots,\#\mathcal{S}_i\}^{N_i}.$ Thus, every point in $\mathcal{T}_k$ can be uniquely represented by a finite word $\underline{p}=(\underline{p}_1,\dots,\underline{p}_k).$
\begin{lem}\label{lem 3.9}
	Given $z=z(\underline{p}_1,\dots,\underline{p}_k)\in\mathcal{T}_k$, for all $i\in\{1,\dots,k\}$ and all $l\in\{1,\dots,N_i\}$ we have that
	$$d_{n_i}(x_{p_l^i}^i,f^{t_{i-1}+m_i+(l-1)(m_i+n_i)}z)<\epsilon_0.$$
\end{lem}
\begin{proof}
	Given $i\in\{1,\dots,k\}$ and $l\in\{1,\dots,N_i\}.$ For $m\in\{1,\dots,k-1\}$, let $z_m=z(\underline{p}_1,\dots,\underline{p}_m)\in \mathcal{T}_m$. Let $a=t_{i-1}+m_i,\ b=(l-1)(m_i+n_i)$. Then,
	\begin{align*}
	d_{n_i}(x_{p_l^i}^i,f^{a+b}z)&\leq d_{n_i}(x_{p_l^i}^i,f^{b}y{\underline{p}_i})+d_{n_i}(f^by_{\underline{p}_i}^i,f^{a+b}z)+d_{n_i}(f^{a+b}z_i,f^{a+b}z)\\
&<\frac{\epsilon_0}{2^{i+5}}+d_{c_i}(y_{\underline{p}_i}^i,f^az)+d_{t_i}(z_i,z)\\
&<\frac{\epsilon_0}{2^{i+5}}+\frac{\epsilon_0}{2^{i+6}}+d_{t_i}(z_i,z_{i+1})+\cdots+d_{t_i}(z_{k-1},z_k)\\
&<\frac{\epsilon_0}{2^{i+5}}+\frac{\epsilon_0}{2^{i+6}}+\frac{\epsilon_0}{2^{i+6}}+\frac{\epsilon_0}{2^{i+7}}+\cdots+\frac{\epsilon_0}{2^{k+5}}\\
&<\sum_{m=1}^{k}\frac{\epsilon_0}{2^{m+5}}+\frac{\epsilon_0}{2^{i+6}}<\epsilon_0.
\end{align*}
\end{proof}

\begin{lem}\label{lem 3.10}
	Under the above conditions, $F\subset I_{\varphi},$ which means that the sequence $\{\frac{1}{n}\sum_{i=0}^{n-1}\varphi(f^ix)\}$ diverges.
\end{lem}

\begin{proof}
	For any $x\in F$, we only need to show $\lim\limits_{k\to\infty}\left|\frac{1}{t_k}S_{t_k}\varphi(x)-\alpha_{\rho(k)}\right|=0.$ Thus, we need to estimate $\left|S_{t_k}\varphi(x)-t_k\alpha_{\rho(k)}\right|$. We can divide the estimation into 3 steps.
	
	$\mathbf{Step\ 1.}$ Estimation on $\mathcal{C}_k$ for $k\ge1.$

	Supposing $y\in\mathcal{C}_k$, let us estimate $\left|\sum_{p=0}^{c_k-1}\varphi(f^py)-c_k\alpha_{\rho(k)}\right|$. By the construction of $\mathcal{C}_k$, there exists $(i_1,\dots,i_{N_k})\in(1,\dots,\#\mathcal{S}_k)^{N_k}$  and $x_{i_j}^k\in\mathcal{S}_k $ satisfying
	$$d_{n_k}(x_{i_j}^k,f^{a_j}y)<\frac{\epsilon_0}{2^{k+5}}\ for\ j=1,\dots,N_k.$$
	Since 
	\begin{align*}
	[0,c_k-1]&=[0,N_kn_k+(N_k-1)m_k-1]\\
	&=\bigcup_{j=1}^{N_k}[a_j,a_j+n_k-1]\cup\bigcup_{j=1}^{N_k-1}[a_j+n_k,a_j+n_k+m_k-1],
\end{align*}
on $[a_j,a_j+n_k-1]$, we have 
\begin{align*}
\left|\sum_{p=0}^{n_k-1}\varphi(f^{a_j+p}y)-n_k\alpha_{\rho(k)}\right|\leq &\left|\sum_{p=0}^{n_k-1}\varphi(f^{a_j+p}y)-\sum_{p=0}^{n_k-1}\varphi(f^px_{i_j}^k)\right|\\
&+\left|\sum_{p=0}^{n_k-1}\varphi(f^px_{i_j}^k)-n_k\alpha_{\rho(k)}\right|\\
&\leq n_kVar\left(\varphi,\frac{\epsilon_0}{2^{k+5}}\right)+4n_k\delta_k;
\end{align*}
on $[a_j+n_k,a_j+n_k+m_k-1]$, we have
$$\left|\sum_{p=0}^{m_k-1}\varphi(f^{a_j+n_k+p}y)-m_k\alpha_{\rho(k)}\right|\leq m_k(\|\varphi\|+\alpha_{\rho(k)})\leq 2m_k\|\varphi\|.$$
Combining these inequalities, we have 
$$\left|\sum_{p=0}^{c_k-1}\varphi(f^py)-c_k\alpha_{\rho(k)}\right|\leq N_kn_k\left(Var\left(\varphi,\frac{\epsilon_0}{2^{k+5}}\right)+4\delta_k\right)+2(N_k-1)m_k\|\varphi\|.$$

$\mathbf{Step\ 2.}$ Estimation on $\mathcal{T}_k$ for $k\ge2.$

For $k\ge2$, let us estimate 
$$A_k:=\max_{x\in\mathcal{T}_k}\left|\sum_{p=0}^{t_k-1}\varphi(f^pz)-t_k\alpha_{\rho(k)}\right|$$
For any $z\in\mathcal{T}_k$, there exists $x\in\mathcal{T}_{k-1}$ and $y\in\mathcal{C}_k$ satisfying
$$d_{t_{k-1}}(x,z)<\frac{\epsilon_0}{2^{k+5}},\ d_{c_k}(y,f^{t_k+m_{k-1}}z)<\frac{\epsilon_0}{2^{k+5}}.$$
On $[0,t_{k-1}+m_k-1]$, we have $|\varphi-\alpha_{\rho(k+1)}|\leq 2\|\varphi\|$, while on $[t_{k-1}+m_k,t_k-1]$ we use the specification property and estimation on $\mathcal{C}_k$ to obtain
\begin{align*}
	A_k\leq&2(t_{k-1}+m_k)\|\varphi\|+c_k\cdot Var(\varphi,\frac{\epsilon_0}{2^{k+5}})\\
	&+N_kn_k\left(Var\left(\varphi,\frac{\epsilon_0}{2^{k+5}}\right)+4\delta_k\right)+2(N_k-1)m_k\|\varphi\|.
\end{align*}
By the choice of $N_k\ and\ n_k$, wh have
$$\frac{t_{k-1}+m_k}{N_k}\to0,\ \frac{(N_k-1)m_k}{N_kn_k}\to0\ as\ k\to\infty.$$
Therefore, taking $k\to\infty$ we have 
$$\frac{A_k}{t_k}\leq\frac{2(t_{k-1}+m_k)\|\varphi\|}{N_k}+2Var\left(\varphi,\frac{\epsilon_0}{2^{k+5}}\right)+4\delta_k+\frac{2(N_k-1)m_k\|\varphi\|}{N_kn_k}\to0.$$

$\mathbf{Step\ 3.}$ Estimation on $F$.

For any $x\in F$, $x\in F_k=\cup_{x\in\mathcal{T}_k}\overline{B}_{t_k}\left(x,\frac{\epsilon_0}{2^{k+5}}\right)$, there exists $z_k(x)\in\mathcal{T}_k$ such that $x\in\overline{B}_{t_k}\left(x,\frac{\epsilon_0}{2^{k+5}}\right)$. Thus,
\begin{align*}
	&\left|S_{t_k}\varphi(x)-t_k\alpha_{\rho(k)}\right|\\
	&\leq\left|S_{t_k}\varphi(x)-S_{t_k}\varphi(z_k(x))\right|+\left|S_{t_k}\varphi(z_k(x))-t_k\alpha_{\rho(k)}\right|\\
	&\leq t_k\cdot Var\left(\varphi,\frac{\epsilon_0}{2^{k+5}}\right)+A_k.
\end{align*}
Dividing the both by $t_k$ and taking $k\to\infty$, we have 
$$\lim\limits_{k\to\infty}\left|\frac{1}{t_k}S_{t_k}\varphi(x)-\alpha_{\rho(k)}\right|=0.$$
The proof has been completed.
\end{proof}

\subsection{Construct a suitable measure on $F$} We now define the measure on $F$ which satisfies the generalized pressure distribution principle.

For each $k$ and every $z=z(\underline{p}_1,\dots,\underline{p}_k)\in\mathcal{T}_k$, we define $\mathcal{L}(z):=\mathcal{L}(\underline{p}_1)\dots\mathcal{L}(\underline{p}_k)$ and $\underline{p}_i=(p_1^i,\dots,p_1^{N_i})\in\{1,\dots,\#\mathcal{S}_i\}^{N_i}$. Let
$$\mathcal{L}(\underline{p}_i):=\prod_{l=1}^{N_i}\exp\left\{S_{n_i}\psi(x_{p_l^i}^i)\log{\frac{1}{5\epsilon_0}}\right\},\ \nu_k:=\sum_{x\in\mathcal{T}_k}^{}\delta_z\mathcal{L}(z).$$ 
Normalizing $\nu_k$ to obtain a sequence of probability measures $\mu_k$. Let
$$\kappa_k:=\sum_{x\in\mathcal{T}_k}^{}\mathcal{L}(z),\ \mu_k:=\frac{1}{\kappa_k}\nu_k.$$

\begin{lem}\label{lem 3.11}
	$\kappa_k=\prod_{i=1}^{k}M_i^{N_i}.$
\end{lem}
\begin{proof}
	We note that
	\begin{align*}
		\kappa_k=\sum_{x\in\mathcal{T}_k}^{}\mathcal{L}(z)&=\sum_{\underline{p}_1\in\{1,\dots,\#\mathcal{S}_1\}^{N_1}}^{}\dots\sum_{\underline{p}_k\in\{1,\dots,\#\mathcal{S}_k\}^{N_k}}^{}(\mathcal{L}(\underline{p}_1)\cdots\mathcal{L}(\underline{p}_1))\\
		&=\left(\sum_{\underline{p}_1\in\{1,\dots,\#\mathcal{S}_1\}^{N_1}}^{}\mathcal{L}(\underline{p}_1)\right)\cdots\left(\sum_{\underline{p}_1\in\{1,\dots,\#\mathcal{S}_k\}^{N_k}}^{}\mathcal{L}(\underline{p}_k)\right).
	\end{align*}
From the definition of $\mathcal{L}(\underline{p}_i)$, for every $i$ we have
$$\left(\sum_{\underline{p}_i\in\{1,\dots,\#\mathcal{S}_i\}^{N_i}}^{}\mathcal{L}(\underline{p}_1)\right)=\prod_{l=1}^{N_i}\left\{\sum_{p^i_l=1}^{\#\mathcal{S}_i}\exp\left\{S_{n_i}\psi(x_{p_l^i}^i)\log{\frac{1}{5\epsilon_0}}\right\}\right\}=M_i^{N_i}.$$
Hence, $\kappa_k=\prod_{i=1}^{k}M_i^{N_i}.$
\end{proof}

\begin{lem}\label{lem 3.12}\cite[Lemma 3.10]{th10}
	Suppose $\nu$ is an accumulation point of the sequence of probability measures $\mu_k$. Then $\nu(F)=1$.
\end{lem}
\begin{proof}
Let $\nu=\lim\limits_{k\to\infty}\mu_{l_k}$ for some $l_k\to\infty$. For any fixed $l$ and all $p\ge0$, since $\mu_{l+p}(F_{l+p})=1$ and $F_{l+p}\subset F_l$, we have $\mu_{l+p}(F_l)=1$. Thus, $\nu(F_l)\ge\limsup_{k\to\infty}\mu_{l_k}(F_{l})=1$. It implies that 
$$\nu(F)=\lim\limits_{l\to\infty}\nu(F_l)=1.$$
\end{proof}

Let $\mathcal{B}:=B_n(q,\epsilon_0/2)$ be an arbitrary ball which intersects $F$. There exists unique $k$ that satisfies $t_k\leq n<t_{k+1}$ and unique $j\in\{0,\dots,N_{k+1}-1\}$ that satisfies $t_k+j(n_{k+1}+m_{k+1})\leq n<t_k+(j+1)(n_{k+1}+m_{k+1}).$ Thus, we have the following lemma which reflects the fact that the number of points in $\mathcal{B}\cap\mathcal{T}_{k+1}$ is restricted.
\begin{lem}\label{lem 3.13}
	Suppose $\mu_{k+1}(\mathcal{B})>0$, then there exists a unique $x\in\mathcal{T}_k$ and $i_1,\dots,i_j\in\{1,\dots,\#\mathcal{S}_{k+1}\}$ satisfying
	$$\nu_{k+1}(\mathcal{B})\leq\mathcal{L}(x)\left(\prod_{l=1}^{j}\exp\left(S_{n_{k+1}}\psi(x_l^{k+1})\log{\frac{1}{5\epsilon_0}}\right)\right)M_{k+1}^{n_{k+1}-j}.$$
\end{lem}
\begin{proof}	
	Since we suppose $\mu_{k+1}(\mathcal{B})>0$, then $\mathcal{T}_{k+1}\cap\mathcal{B}\neq\emptyset.$ Let $z_1=z(x_1,y_1),\ z_2=z(x_2,y_2)\in\mathcal{T}_k\cap\mathcal{B}$, where $x_1,x_2\in\mathcal{T}_k\ and\ y_1,y_2\in\mathcal{C}_{k+1}.$ Let $y_1=y(i_1,\dots,i_{N_{k+1}}),\ y_2=y(l_1,\dots,l_{N_{k+1}}).$ We have
	\begin{align*}
	d_{t_k}(x_1,x_2)&\leq d_{t_k}(x_1,z_1)+d_{t_k}(z_1,z_2)+d_{t_k}(z_2,x_2)\\
&<\frac{\epsilon_0}{2^{k+6}}+\epsilon_0+<\frac{\epsilon_0}{2^{k+6}}<\frac{33\epsilon_0}{32}
\end{align*}
and thus we have a contradiction with Lemma \ref{lem 3.7}. Similarly, we prove that $i_t=l_t$ for $t=\{1,\dots,j\}.$ Suppose there exists $t,\ 1\leq t\leq j$, such that $i_t\neq l_t$. Since the specification property, we have 
$$d_{n_{k+1}}(x_{i_t}^{k+1},f^{a_t}y_1)<\frac{\epsilon_0}{2^{k+5}},\ d_{n_{k+1}}(x_{l_t}^{k+1},f^{a_t}y_2)<\frac{\epsilon_0}{2^{k+5}},$$
$$d_{c_{k+1}}(y_1,f^{t_k+m_{k+1}}z_1)<\frac{\epsilon_0}{2^{k+6}},\ d_{c_{k+1}}(y_1,f^{t_k+m_{k+1}}z_1)<\frac{\epsilon_0}{2^{k+6}}.$$
Thus, 
\begin{align*}
	d_{n_{k+1}}(x_{i_t}^{k+1},x_{l_t}^{k+1})&\leq d_{n_{k+1}}(x_{i_t}^{k+1},f^{a_t}y_1)+d_{c_{k+1}}(y_1,f^{t_k+m_{k+1}}z_1)\\
	&+d_n(z_1,z_2)+d_{c_{k+1}}(y_2,f^{t_k+m_{k+1}}z_2)+d_{n_{k+1}}(x_{l_t}^{k+1},f^{a_t}y_2)\\
	&<\epsilon_0+\frac{\epsilon_0}{2^{k+5}}\cdot4<\frac{9\epsilon_0}{8},
\end{align*}
which contradicts the fact that $\mathcal{S}_{k+1}$ is $\left(n_{k+1},\frac{9\epsilon_0}{8}\right)$ separated.
Since $x$ and $(i_1,\dots,i_j)$ is the same for all points $z=z(x,y)$, $y=y(i_1,\dots,i_{N_{k+1}})$ which lies in $\mathcal{T}_{k+1}\cap\mathcal{B}$, we can conclude that there are at most $M_{k+1}^{N_{k+1}-j}$ such pints. Then,
\begin{align*}
\nu_{k+1}(\mathcal{B})&\leq\mathcal{L}(x)\sum_{\underline{p}_{k+1}}^{}\mathcal{L}(\underline{p}_{k+1})\\
&\leq\mathcal{L}(x)\left(\prod_{l=1}^{j}\exp\left(S_{n_{k+1}}\psi(x_{i_l}^{k+1})\log{\frac{1}{5\epsilon_0}}\right)\right)M_{k+1}^{n_{k+1}-j}.
\end{align*}
\end{proof}

\begin{lem}\label{lem 3.14}
	Let $x\in\mathcal{T}_k$ and $i_1,\dots,i_j$ be as before. Then 
	\begin{align*}
		&\mathcal{L}(x)\prod_{l=1}^{j}\exp\left(S_{n_k+1}\psi(x_{i_l}^{k+1})\cdot\log\frac{1}{5\epsilon_0}\right)\\
	&\leq\exp\left(\left(S_n\psi(q)+2nVar(\psi,\epsilon_0)+\left(jm_{k+1}+\sum_{i=1}^{k}N_im_i\right)\|\psi\|\right)\log{\frac{1}{5\epsilon_0}}\right).
	\end{align*}
\end{lem}
\begin{proof}
	Let $x:=x(\underline{p}_1,\dots,\underline{p}_k)$. From Lemma \ref{lem 3.9}, we have 
	$$d_{n_i}(x_{p_l^i}^i,f^{t_{i-1}+m_i+(l-1)(m_i+n_i)}x)<\epsilon_0$$
	for $i\in\{1,\dots,k\}$ and $l\in\{1,\dots,N_i\}$. It follows that
	\begin{align*}
		S_{n_i}\psi(x_{p_l^i}^i)&\leq S_{n_i}\psi(x_{p_l^i}^i)-S_{n_i}\psi(f^{t_{i-1}+m_i+(l-1)(m_i+n_i)}x)\\
		&+S_{n_i}\psi(f^{t_{i-1}+m_i+(l-1)(m_i+n_i)}x)\\
		&\leq n_iVar(\psi,\epsilon_0)+S_{n_i}\psi(f^{t_{i-1}+m_i+(l-1)(m_i+n_i)}x)
	\end{align*}
	and
	\begin{align*}
		S_{n_{k+1}}\psi(x_{i_l}^{k+1})&\leq S_{n_{k+1}}\psi(x_{i_l}^{k+1})-S_{n_{k+1}}\psi(f^{t_k+m_{k+1}+(l-1)(n_{k+1}+m_{k+1})}z)\\
		&+S_{n_{k+1}}\psi(f^{t_k+m_{k+1}+(l-1)(n_{k+1}+m_{k+1})}z)\\
		&\leq n_{k+1}Var\left(\psi,\frac{\epsilon_0}{2^{k+6}}\right)+S_{n_{k+1}}\psi(f^{t_k+m_{k+1}+(l-1)(n_{k+1}+m_{k+1})}z).
	\end{align*}
Thus we have
\begin{align*}
&\sum_{i=1}^{k}\sum_{l=1}^{N_i}S_{n_i}\psi(x_{p_l^i}^i)\cdot\log{\frac{1}{5\epsilon_0}}\\
\leq&\left\{\sum_{i=1}^{k}\sum_{l=1}^{N_i}n_iVar(\psi,\epsilon_0)+\sum_{i=1}^{k}\sum_{l=1}^{N_i}S_{n_i}\psi(f^{t_{i-1}+m_i+(l-1)(m_i+n_i)}x)\right\}\cdot\log{\frac{1}{5\epsilon_0}}\\
\leq&\left\{t_kVar(\psi,\epsilon_0)+S_{t_{k}}\psi(x)+\sum_{i=1}^{k}N_im_i\|\psi\|\right\}\cdot\log{\frac{1}{5\epsilon_0}}
\end{align*}
and
\begin{align*}
	\sum_{l=1}^{j}S_{n_{k+1}}\psi(x_{i_l}^{k+1})\leq(n-t_k)Var\left(\psi,\frac{\epsilon_0}{2^{k+6}}\right)+S_{n-t_k}\psi(f^{t_k}z)+jm_{k+1}\|\psi\|.
\end{align*}
Since $d_{t_k}(x,q)\leq d_{t_k}(x,z)+d_{t_k}(q,z)<\epsilon_0$, we have
\begin{align*}
S_{t_{k}}\psi(x)+S_{n-t_k}\psi(f^{t_k}z)&=S_{t_{k}}\psi(x)-S_{t_{k}}\psi(q)+S_{n-t_k}\psi(f^{t_k}z)\\
&-S_{n-t_k}\psi(f^{t_k}q)+S_n\psi(q)\\
&\leq t_kVar(\psi,\epsilon_0)+(n-t_k)Var(\psi,\epsilon_0)+S_n\psi(q).
\end{align*}
Combining the above inequalities, we have that
\begin{align*}
	&\mathcal{L}(x)\prod_{l=1}^{j}\exp\left(S_{n_k+1}\psi(x_{i_l}^{k+1})\cdot\log\frac{1}{5\epsilon_0}\right)\\
&\leq\exp\left\{\left(S_n\psi(q)+2nVar(\psi,\epsilon_0)+\left(jm_{k+1}+\sum_{i=1}^{k}N_im_i\right)\|\psi\|\right)\log{\frac{1}{5\epsilon_0}}\right\}.
\end{align*}
\end{proof}

Similarly, we give the following Lemma without proof which shows that the points contained in $\mathcal{B}\cap\mathcal{T}_{k+p}$ are restricted either.
\begin{lem}\label{lem 3.15}
	For any $p\ge1$, suppose $\mu_{k+p}(\mathcal{B})>0$. Let $x\in\mathcal{T}_k$ and $i_1,\dots,i_j$ be as before. Then every $x\in\mathcal{B}\cap\mathcal{T}_{k+p}$ descends from some point in $\mathcal{T}_k\cap\mathcal{B}$. We have
	$$\nu_{k+p}(\mathcal{B})\leq\mathcal{L}(x)\left\{\prod_{l=1}^{j}\exp\left(S_{n_{k+1}}\psi(x_l^{k+1})\log{\frac{1}{5\epsilon_0}}\right)\right\}M_{k+1}^{n_{k+1}-j}\cdots M_{k+p}^{n_{k+p}}.$$
\end{lem}
Since $\mu_{k+p}=\frac{1}{\kappa_{k+p}}\nu_{k+p}$ and $\kappa_{k+p}=\kappa_kM_{k+1}^{N_{k+1}}\cdots M_{k+p}^{N_{k+p}}$, immediately, we have 
\begin{align*}
&\mu_{k+p}(\mathcal{B})\leq\frac{1}{\kappa_kM_{k+1}^j}\mathcal{L}(x)\left\{\prod_{l=1}^{j}\exp\left(S_{n_{k+1}}\psi(x_l^{k+1})\log{\frac{1}{5\epsilon_0}}\right)\right\}\\
&\leq\frac{1}{\kappa_kM_{k+1}^j}\exp\left\{\left(S_n\psi(q)+2nVar(\psi,\epsilon_0)+\left(jm_{k+1}+\sum_{i=1}^{k}N_im_i\right)\|\psi\|\right)\log{\frac{1}{5\epsilon_0}}\right\}.
\end{align*}
\begin{lem}\label{lem 3.16}
	For sufficiently large $n$, we have 
	$$\kappa_kM_{k+1}^j\ge\exp\left(\left(\inf_{diam\xi<5\epsilon_0}h_{\mu_{\rho(k)}}(f,\xi)+\int\psi d{\mu_{\rho(k)}}\cdot\log{\frac{1}{5\epsilon_0}}-5\gamma\right)n\right)$$
\end{lem}
\begin{proof}
	Since 
	$$M_k\ge\exp\left\{n_k\left(\inf_{diam\xi<5\epsilon_0}h_{\mu_{\rho(k)}}(f,\xi)+\int\psi d{\mu_{\rho(k)}}\cdot\log{\frac{1}{5\epsilon_0}}-4\gamma\right)\right\},$$
	we note that $C=\inf_{diam\xi<5\epsilon_0}h_{\mu_{\rho(k)}}(f,\xi)+\int\psi d{\mu_{\rho(k)}}\cdot\log{\frac{1}{5\epsilon_0}}$ and we have
	\begin{align*}
		&\kappa_kM_{k+1}^j\\
		&=M_1^{N_1}\cdots M_k^{N_k}M_{k+1}^j\ge\exp\left((C-4\gamma)(n_1N_1+\cdots+n_kN_k+n_{k+1}j)\right)\\
		&\ge\exp\left((C-5\gamma)((n_1+m_1)N_1+\cdots+(n_k+m_k)N_k+(n_{k+1}+m_{k+1})(j+1))\right)\\
		&=\exp\left((C-5\gamma)(t_k+(n_{k+1}+m_{k+1})(j+1))\right)\\
		&\ge\exp\left((C-5\gamma)n\right).
	\end{align*}
\end{proof}

\begin{lem} \label{lem 3.17}
	For sufficiently large $n$, we have 
	$$\limsup_{k\to\infty}\mu_k\left(B_n\left(q,\frac{\epsilon_0}{2}\right)\right)\leq\exp\left\{-n(C-7\gamma-Var(\psi,\epsilon_0))+S_n\psi(q)\cdot\log{\frac{1}{5\epsilon_0}}\right\}.$$
\end{lem}
\begin{proof}
	For sufficiently large $n$, and any $p>1$,
	\begin{align*}
		&\mu_{k+p}(\mathcal{B})\\
		&\leq\frac{1}{\kappa_kM_{k+1}^j}\exp\left\{\left(S_n\psi(q)+2nVar(\psi,\epsilon_0)+\left(jm_{k+1}+\sum_{i=1}^{k}N_im_i\right)\|\psi\|\right)\log{\frac{1}{5\epsilon_0}}\right\}\\
		&\leq\frac{1}{\kappa_kM_{k+1}^j}\exp\left\{S_n\psi(q)\cdot\log{\frac{1}{5\epsilon_0}}+n(2Var(\psi,\epsilon_0)+\gamma)\right\}\\
		&\leq\exp\left\{-n(C-5\gamma)+S_n\psi(q)\cdot\log{\frac{1}{5\epsilon_0}}+n(2Var(\psi,\epsilon_0)+\gamma)\right\}\\
		&=\exp\left\{-n(C-6\gamma-2Var(\psi,\epsilon_0))+S_n\psi(q)\cdot\log{\frac{1}{5\epsilon_0}}\right\}.
	\end{align*}
	We arrive the second inequality is because $n_k\gg m_k$, thus
	$$\frac{jm_{k+1}+\sum_{i=1}^{k}N_im_i}{n}\leq\frac{jm_{k+1}+\sum_{i=1}^{k}N_im_i}{t_k+j(n_{k+1}m_{k+1})}\to0,\ \text{as}\ k\to\infty.$$
\end{proof}
Now we give the generalized pressure distribution principle which is a modification of \cite[Proposition 3.2]{th10}.
\begin{prop}\label{prop 3.18}
	Let $f:X\to X$ be a continuous transformation and $\epsilon>0$. For $Z\subset X$ and a constant $s\ge0$, suppose there exist a constant $C>0$, a sequence of Borel probability measure $\mu_k$ and integer $N$ satisfying
	$$\limsup_{k\to\infty}\mu_k\left(B_n\left(q,\frac{\epsilon}{2}\right)\right)\leq C\exp\left\{-sn+S_n\psi(x)\log{\frac{1}{5\epsilon}}\right\}$$
	for every $B_n(x,\epsilon)$ such that $B_n\left(x,\epsilon/2\right)\cap Z\neq\emptyset$ and $n\ge N$. Furthermore, assume that at least one accumulate point $\nu$ of $\mu_k$ satisfies $\nu(Z)>0$. Then $M_{\epsilon/2}\left(f,Z,d,\psi\right)\ge s.$
\end{prop}
\begin{proof}
	Let $\nu$ and $\epsilon>0$ satisfy the conditions, and $\mu_{k_j}$ be the sequence of measures which converges to $\nu$. Let $\Gamma=\{B_{n_i}(x_i,\epsilon/2)\}_{i\in I}$ cover Z with all $n_i\ge N$. We can assume that $B_{n_i}(x,\epsilon/2)\cap Z\neq\emptyset$ for every $i.$ Then
	\begin{align*}
		\sum_{i\in I}^{}\exp\left\{-sn_i+S_{n_i}\psi(x)\log{\frac{2}{\epsilon}}\right\}&\ge\sum_{i\in I}^{}\exp\left\{-sn_i+S_{n_i}\psi(x)\log{\frac{1}{5\epsilon}}\right\}\\
		&\ge\frac{1}{C}\sum_{i\in I}^{}\limsup_{k\to\infty}\mu_k\left(B_{n_i}(x_i,\epsilon)\right)\\
		&\ge\frac{1}{C}\sum_{i\in I}^{}\liminf_{j\to\infty}\mu_{k_j}\left(B_{n_i}(x_i,\epsilon)\right)\\
&\ge\frac{1}{C}\sum_{i\in I}^{}\nu\left(B_n(x_i,\epsilon)\right)\\
&\ge\frac{1}{C}\nu(Z)>0.
	\end{align*}
	Thus, we conclude that $m_{\epsilon/2}(f,Z,s,d,\psi)>0$ and $M_{\epsilon/2}(f,Z,d,\psi)\ge s.$
\end{proof}

By Lemma \ref{lem 3.10} we have that $M_{\epsilon_0}(f,I_{\varphi},d,\psi)\ge M_{\epsilon_0}(f,F,d,\psi)$ and by Lemma \ref{lem 3.17} we have that $M_{\epsilon_0/2}(f,K_{\alpha},d,\psi)\ge C-7\gamma-2Var(\psi,\epsilon_0)$. Combining all the above lemma, we have that
\begin{align*}
	S-\gamma&\leq\frac{\inf_{diam\xi<5\epsilon_0}h_{\mu_{\rho(k)}}(f,\xi)+\Big(\log{\frac{1}{5\epsilon_0}}\Big)\int\psi \rm{d}{\mu_{\rho(k)}}}{\log{\frac{1}{5\epsilon_0}}}\\
	&\leq\frac{M_{\frac{\epsilon_0}{2}}(f,I_{\varphi},d,\psi)+6\gamma+2Var(\psi,\epsilon_0)}{\log{\frac{2}{\epsilon_0}}}\cdot\frac{\log{\frac{2}{\epsilon_0}}}{\log{\frac{1}{5\epsilon_0}}}\\
	&\leq\left\{\overline{\rm{mdim}}_M^B(f,I_{\varphi},d,\psi)+\frac{7\gamma+2Var(\psi,\epsilon_0)}{\log{\frac{2}{\epsilon_0}}}\right\}.
\end{align*}
As $\gamma>0$ is arbitrary and $\gamma\to0\Rightarrow\epsilon_0\to0$, we obtain  
$$S\leq\overline{\rm{mdim}}_M^B(f,I_{\varphi},d,\psi).$$ 
Finally, we have finished the proof.

\section*{Acknowledgement} 
The  second author was supported by the
National Natural Science Foundation of China (No.12071222).
The third authors was  supported by the
National Natural Science Foundation of China (No.11971236), Qinglan Project of Jiangsu Province of China.  The work was also funded by the Priority Academic Program Development of Jiangsu Higher Education Institutions.  

\section*{Data availability} 
No data was used for the research described in the article.
\section*{Conflict of interest} 
The author declares no conflict of interest.

\end{document}